\documentclass[12 pt]{amsart}

\usepackage{amssymb}
\usepackage{latexsym}
\usepackage{url}
\usepackage{graphicx}
\usepackage{gauss}

\DeclareMathOperator{\Tr}{Tr}

\newtheorem{thm}{Theorem}[section]
\newtheorem{cor}[thm]{Corollary}
\newtheorem{rmrk}[thm]{Remark}
\newtheorem{lem}[thm]{Lemma}
\newtheorem{prop}[thm]{Proposition}

\newtheorem{defin}[thm]{Definition}
\newtheorem{expl}[thm]{Example}

\newtheorem{prob_s}{Problem}
\newtheorem{conj_s}[prob_s]{Conjecture}

\begin{document}
\title[Signless Laplacian eigenvector sign patterns]{On the sign patterns of the smallest signless Laplacian eigenvector}
\date{July 29, 2013}

\author{Felix Goldberg}
\address{Hamilton Institute, National University of Ireland Maynooth, Ireland}
\email{felix.goldberg@gmail.com}

\author{Steve Kirkland}
\address{Hamilton Institute, National University of Ireland Maynooth, Ireland} 
\email{stephen.kirkland@nuim.ie}

\begin{abstract}
Let $H$ be a connected bipartite graph, whose signless Laplacian matrix is $Q(H)$. Suppose that the bipartition of $H$ is $(S,T)$ and that $x$ is the eigenvector of the smallest eigenvalue of $Q(H)$. It is well-known that $x$ is positive and constant on $S$, and negative and constant on $T$.

The resilience of the sign pattern of $x$ under addition of edges into the subgraph induced by either $S$ or $T$ is investigated and a number of cases in which the sign pattern of $x$ persists are described.
\end{abstract}

%%\begin{keywords}
%%Signless Laplacian matrix, Eigenvector signs, Bipartite graph, Maximal independent set.
%%\end{keywords}
%%\begin{AMS}
%%05C50,15A18,15B48. 
%%\end{AMS}


\thanks{{S.K.'s Research supported in part by the Science Foundation Ireland under Grant No.
SFI/07/SK/I1216b.}}

\maketitle

\section{Introduction}
%We deal with the signless Laplacian matrix of a graph $G$.

Let $G$ be a graph with adjacency matrix $A(G)$ and let $D(G)$ be the diagonal matrix of the vertex degrees of $G$. The Laplacian matrix of $G$ is $L(G)=D(G)-A(G)$ and the \emph{signless Laplacian matrix} of $G$ is $Q(G)=D(G)+A(G)$. The matrix $Q(G)$ has been largely subject to benign neglect up to very recent times, but it has received a lot of attention lately - some of the results of which are summarized in the surveys \cite{CveRowSim07,Signless1,Signless2,Signless3}.

It is well-known (cf. \cite[pp. 157-8]{CveRowSim07}) that $0$ is an eigenvalue of $Q(G)$ with multiplicity equal to the number of bipartite connected components of $G$.
We shall denote the smallest eigenvalue of $Q(G)$ by $\mu(G)$.

In this note we consider the relationship between bipartiteness and the signless Laplacian from a slightly different angle, studying the sign pattern of an eigenvector that corresponds to $\mu(G)$. Note also that we will restrict our attention to \emph{real} eigenvectors.

%Before we state our main motivating question, a word about notation: if $G$ is a graph with labelled vertex set $V(G)=\{v_{1},\ldots,v_{n}\}$, we shall use the canonical identification between a vector $x \in \mathbb{R}^{n}$ and a function $f_{x}:V(G) \rightarrow \mathbb{R}$ that operates as $f_{x}(v_{i})=x_{i}$. Using this langauge, if $S \subseteq V(G)$, then $x(S)$ means a subvector of $x$ formed by the indices that correspond to $S$. And finally, $x>0$ will mean that all entries of $x$ are strictly positive.

A few words about notation: if $V(G)$ is labelled as $\{1,2,\ldots,n\}$ and $x \in \mathbb{R}^{n}$, then for any nonempty subset $S \subseteq V(G)$ we shall mean by $x(S)$ the vector in $\mathbb{R}^{|S|}$ formed by deleting from $x$ all entries not corresponding to elements of $S$. The all-ones vector of length $n$ will be denoted by $\mathbf{1}_{n}$ or just $\mathbf{1}$ if the length is clear from the context. We write $v>0$ to indicate that all the entries of a vector $v$ are strictly positive.

%\subsection{Eigenvectors of $\mu$}
The following fact is well-known:
\begin{prop}\label{prop:basic}
Let $H$ be a connected biparite graph with bipartition $(S,T)$. For any eigenvector $x$ corresponding to $\mu(H)$ there is a nonzero number $c$ so that:
$$
x(S)=c \mathbf{1}_{|S|} ,x(T)=-c \mathbf{1}_{|T|}.
$$
\end{prop}

The main result of the paper \cite{Roth89} by Roth can be seen as an interesting generalization of Proposition \ref{prop:basic}:
\begin{prop}\label{prop:roth89}
Let $H$ be a connected biparite graph with bipartition $(S,T)$. Let $D$ be any diagonal matrix and let $x$ be an eigenvector corresponding to the smallest eigenvalue of $Q(H)+D$. Then:
\begin{equation}\label{eq:roth_eq}
x(S)>0, \quad x(T)<0,
\end{equation}
or vice versa.
\end{prop}

We are interested in generalizing Proposition \ref{prop:basic}, showing that (\ref{eq:roth_eq}) continues to hold even when edges are added on one of the sides of $H$. We keep $D=0$, however. Let us make the following definition:

\begin{defin}\label{defin:roth}
Let $H$ be a connected graph and let $S \subseteq V(H)$ be a maximal independent set. We say that $H$ is \emph{$S$-Roth} if for every eigenvector $x$ corresponding to $\mu(H)$ we have that $$x(S)>0, \quad x(V(H)-S)<0,$$ or vice versa.
\end{defin}

%\begin{rmrk}
The assumption that $S$ is a maximal independent set is made in order to rule out the uninteresting case when $H$ is bipartite and $S$ is a proper subset of one of the partite sets associated with $H$. In that case there is a smallest eigenvector of $Q(H)$ that is positive on $S$ but has mixed signs on the complement of $S$.
%\end{rmrk}

\begin{rmrk}
Notice that if $H$ is $S$-Roth, then $\mu(H)$ must be a simple eigenvalue.
\end{rmrk}

%In these terms, the well-known property of bipartite graphs that serves as our point of departure can be phrased as:
%We can now phrase Proposition \ref{prop:basic} thus:

%The following is immediate:

Proposition \ref{prop:basic} can now be stated as:

\begin{prop}
Let $H$ be a connected biparite graph with bipartition $(S,T)$. Then $H$ is $S$-Roth.
\end{prop}
%\begin{prop}
%For any connected bipartite graph $B$, the $(B,\overline{K_{t}})$ is a Roth pair.
%\end{prop}

%In the rest of the paper we shall prove that various classes of graphs are $S$-Roth. For instance, we show in Corollary \ref{cor:st} that any $H$ that is a join of the independent set $S$ and another graph $T$ is $S$-Roth, provided only that $|S| \geq |T|$.

\subsection{Plan of the paper}
Section \ref{sec:motivate} will be devoted to a discussion of the motivation behind the concept of a $S$-Roth graph. In sections \ref{sec:facts} and \ref{sec:tools} we marshal the definitions and tools requisite to the study $S$-Roth graphs. 

Armed thus, in Sections \ref{sec:basic} and \ref{sec:gamma} we analyze the definition and obtain a combinatorial condition, Theorem \ref{thm:harmcond}, that ensures that a given graph is $S$-Roth. The next two sections discuss and illustrate this result: in Section \ref{sec:yeast} we show how it applies to a graph from \cite{Protein}, while in Section \ref{sec:case} we undertake an exhaustive study of some families of special cases, in order to indicate both the power of Theorem \ref{thm:harmcond} and its limitations.
% outlining as well possible directions for further work. 

The following sections, \ref{sec:kst} through \ref{sec:cycle} take up the case of a complete bipartite graph to one of whose sides edges are added. In a sense, this is the ``ideal case" of the problem on protein networks outlined in Section \ref{sec:noise}; apart from affording a challenging technical problem, it throws some light on the more general case by indicating that $S$-Rothness can occur in two widely disparate situations - either when we create a dense graph induced on $T$ or when we create a sparse graph there. This insight is applicable to the general case as well. 

The techniques needed to handle the two situations are a bit different: Theorem \ref{thm:Gdeg} which addresses the dense case is derived using the method of Section \ref{sec:basic}, whereas Theorem \ref{thm:deg2} which handles the sparse case require a reappraisal of the problem and the introduction of some more matrix-theoretic machinery (carried out in Sections \ref{sec:q_alt} and \ref{sec:more}, respectively). 

Finally, in Section \ref{sec:open} we suggest some open problems for the further development of the subject.

\section{Motivation}\label{sec:motivate}
The concept of a $S$-Roth graph originally arose as part of an effort to develop algorithms that detect bipartite structures in noisy data. It also weaves together a number of strands of algbraic graph theory and matrix theory.

\subsection{The lock-and-key model for protein networks}
Our description in this subsection is based on the paper \cite{Protein} by Morrison \emph{et al.}
Biologists study networks of protein interactions, representing them as graphs whose vertices are the proteins and connecting vertices by an edge if the corresponding proteins interact. Once the network is constructed, one would like to analyze it in order to detect groups of proteins within which there is strong interactions. Once such a subgroup has been identified, the biologist can study it in more depth, looking for the underlying chemical and biological mechanisms which cause the proteins to interact strongly with another. The graph theorist's job in this case is to facilitate the detection of the subgroup.

Put in this - addmittedly somewhat naive way - the relevant graph-theoretical problem seems to be the detection of cliques. However, it turns out that a biologically more plausible model for protein groupings is the so-called ``lock-and-key" model, introduced in \cite{Protein}. Imagine that a set of keys and a set of locks have been isssued to some of the vertices and that vertices interact if one has a key and the other a lock. We thus get a complete bipartite subgraph as a model for the protein group. (For the digraph version of the problem see \cite{TayVasHig11}).
%and indeed earlier works on this subject (such as \cite{Jeong})

So far goes the theory - in order to actually construct such a graph we need to perform an experimental assay which determines which protein pairs interact with each other and which pairs do not. Since experimental data are always attended by some amount of noise, we must take into account that our graph will have some spurious edges and lack other edges which ought to have been be present. This means that we should in fact be looking for a \emph{nearly-complete almost-bipartite subgraph}. 

The authors of \cite{Protein} have proposed a method to detect such subgraphs using eigenvectors of the adjacency matrix. We feel however that the signless Laplacian matrix may be better attuned to the problem at hand.

\subsubsection{The spectral approach}
Suppose that the graph $G$ has a connected component $C$ isomorphic to $K_{p,q}$ for some $p$ and $q$. Then we have $\mu(G)=0$ and the kernel of $Q(G)$ contains a vector $v$ that is $1$ on one partite class of $C$, $-1$ on the other partite class, and $0$ on the rest of the vertices of $G$. It stands to reason that neither $\mu(G)$ nor $v$ would change much if we performed a few times on $G$ one of the following operations:
\begin{itemize}
\item
Deleting an edge between the two partite classes of $C$.
\item
Adding an edge within a partite class of $C$.
\item 
Adding an edge between $C$ and $G-C$.
\end{itemize}
In biological parlance, $C$ is the protein grouping we are looking for in the large network $G$, and the operations we have just discussed correspond to the experimental noise.

If the sign pattern of $v$ does not change after the application of the noise operations, we can use $v$ in an algorithm that detects $C$ as a nearly-complete almost-bipartite subgraph of $G$. The mathematical problem here is then to quantify this assertion by finding conditions under which the sign pattern of $v$ can be assured to hold (by limiting, say, the number of times we are allowed to apply the noise operations). 

A theorem of this kind can be used in practice to inform the practitioner whether the spectral algorithm applies to his problem, given some knowledge about the level of noise in the experiment. If the noise level does not exceed the threshold guaranteed by the theorem, then the spectral detection algorithm can be safely and profitably applied.

\subsubsection{$S$-Rothness and a model of a noisy graph}
In this paper we focus on the effects of the first two noise operations described above. We will in effect be assuming that $G=C$, so that the third operation is not relevant to our set-up. We will also restrict the second operation to the addition of edges only within one of the partite classes of $G$ (this could correspond to a situation where the experimental machinery is prone to misdetecton of spurious interactions between, say, keys, but is not likely to report a spurious interaction between locks).
What now obtains is exactly the definition of $S$-Rothness. \emph{A graph is $S$-Roth if its underlying bipartite structure can be recovered by examining the eigenvector $v$}.
%Let us now make two simplfiying assumptions, in order to be able to analyze the theoretical model we have introduced above. We will take $G=C$ (that is, we are assuming that we are dealing with 

\subsection{$Q(G)$ and bipartite subgraphs}\label{sec:noise}

Desai and Rao \cite{DesRao94} first tackled the issue of what structural bipartiteness properties can be derived from the fact that the smallest eigenvalue $\mu(G)$ of $Q(G)$ is small but non-zero. They showed that a low value of $\mu(G)$ indicates the presence of a nearly-bipartite subgraph of $G$ that is weakly connected to the rest of a graph. Very recently their results were improved in \cite{FalFan12}. However, neither \cite{DesRao94} nor \cite{FalFan12} gives a way to detect such a subgraph in a given graph $G$. 

A constructive approach was taken in \cite{KirDeb11} where a condition was established for a given subset of $S \subseteq V(G)$ to induce a bipartite subgraph, based on $\mu(G)$ and the Rayleigh quotient for $Q(G)$ of a certain indicator vector of $S$.

\subsection{Perturbation of $\mu(G)$}
The recent paper \cite{HePan12} studies how $\mu(G)$ is affected by small changes in the structure of $G$, such as vertex deletion or edge subdivision. Our work can be seen as a counterpart to theirs by examining instead the changes in the corresponding eigenvector - a more daunting problem.

\subsection{Graph eigenvector sign patterns and nodal domains}
The study of the sign patterns of graph eigenvectors goes back to Fiedler's groundbreaking work \cite{Fie75} on the eigenvector of the second smallest eigenvalue of $L(G)$. Fiedler's work has been extended and generalized to study the \emph{nodal domains} of eigenvectors of graphs (a nodal domain is a maximal connected subgraph whose vertices have the same sign, either in the strong or in the weak sense). For more information about this we point the reader to \cite{DavGlaLeySta01,BiyHorLeyPisSta04}. It is also interesting to note that Roth's theorem (Proposition \ref{prop:roth89} in our notation) has been recently re-interpreted as a theorem on nodal domains in \cite{Oren07}.

%%We would also like to mention the paper \cite{MayOleDri90} that deals with zero

\subsection{Eigenvector sign patterns of general matrices}
Finally, we note that the preservation of the sign pattern of an eigenvector under a perturbation of the underlying matrix has been studied in the matrix-theoretical setting. Stated in such generality, this question poses formidable difficulties and there is relatively little that can be said about it. Nevertheless, we refer the reader to the papers \cite{DeifRohn94} and \cite{Har95} for some interesting results.

\section{Some useful terms and facts}\label{sec:facts}
Terms used without explanation may be found in the book \cite{GraphsDigraphs}. The minimum and maximum degrees of the graph $G$ will be denoted by $\delta(G)$ and $\Delta(G)$, respectively. The degree of a vertex $v$ in a graph $G$ will be denoted by $d_{G}(v)$. To indicate that vertices $i,j \in V(G)$ are adjacent we will employ the notation $ij \in E(G)$ or $i \sim_{G} j$.
A cycle (path) on $n$ vertices will be denoted $C_{n}$ ($P_{n}$).

%First of all, let us recall some basic terms from graph theory: %
The \emph{disjoint union} of two graphs
$G_{1}$,$G_{2}$ will be denoted by $G_{1} \cup G_{2}$, and is the
graph whose vertex and edge sets are the disjoint unions of those
of $G_{1}$ and $G_{2}$. 

The \emph{join}, denoted by $G_{1} \vee
G_{2}$, is obtained from $G_{1} \cup G_{2}$ by adding to it all
edges between vertices in $V(G_{1})$ and vertices in $V(G_{2})$. Finally, the \emph{complement} $\overline{G}$ of a graph $G$ is the graph on the same vertex set whose edges are those and only those not present in $G$.

From these definitions the following simple but very useful fact immediately arises:
\begin{prop}\label{prop:join1}
Let $G$ be any graph. Then $G$ can be written as a join $G=G_{1} \vee G_{2}$ if and only if $\overline{G}$ is disconnected.
\end{prop}

In light of Proposition \ref{prop:join1} we can speak of a \emph{maximal join decomposition} of a graph $G$ as %\begin{equation}\label{eq:join}
$
G=G_{1} \vee G_{2} \vee \ldots \vee G_{k}$,
%\end{equation}
where each $\overline{G}_{i}$ is connected. 
%It is then understood that the vertex set of $V(H)

We also need some facts about the (usual) Laplacian eigenvalues.
The first lemma is an expanded statement of \cite[Proposition 2.3]{Moh97}, including some properties established in the proof.
\begin{lem}\label{lem:mohar}
Let $G$ be a graph on $n$ vertices. Then:
\begin{itemize}
\item
$\mathbf{1}_{n}$ is a zero eigenvector of $L(G)$. 
\item
If $G$ is connected, then all other eigenvectors of $L(G)$ are orthogonal to $\mathbf{1}_{n}$.
\item
Every zero eigenvector of $L(G)$ takes a constant value on each connected component of $G$.
\end{itemize}
\end{lem}

\begin{lem}\cite[Theorem 2.1]{Merris98}\label{lem:join}
If $H$ is a graph on $n$ vertices and $H=G_{1} \vee G_{2}$, then $\lambda_{n}(L(H))=n$.
\end{lem}

\begin{lem}\label{lem:join_max}
Let $H=G_{1} \vee G_{2} \vee \ldots \vee G_{k}$ be a maximal join decomposition of $H$ and assume that $k \geq 2$. Let $n=|V(H)|$. Let $x$ be an eigenvector of $L(H)$ corresponding to $n$. Then $x$ is constant on the vertex set of each $G_{i}$, $i=1,2,\ldots,k$.
\end{lem}
\begin{proof}
The complement $\overline{H}$ is the disjoint union of the complements of the joinees: $\overline{H}=\overline{G}_{1} \cup \overline{G}_{2} \cup \ldots \cup \overline{G}_{k}$.
Observe that $L(H)+L(\overline{H})=nI-J$. Therefore, by the first two parts of Lemma \ref{lem:mohar} we see that if $x$ is an eigenvector of $L(H)$ corresponding to $n$, then $x$ is also a zero eigenvector of $L(\overline{H})$. The conclusion now follows from the third part of Lemma \ref{lem:mohar}.
%It is well-known (cf. the proof of \cite[Proposition 2.3]{Moh97}) that the zero eigenvectors of a Laplacian matrix constant on each connected component. Since $L(H)+L(\overline{H})=nI-J$ we see that the eigenvectors of $L(H)$ corresponding to $n$ are the same as the zero eigenvectors of $L(\overline{H})$ (except for $\mathbf{1}_{n}$ which corresponds to zero in both $H$ and $\overline{H}$).
\end{proof}

%A graph $G$ is called \emph{split} if its vertex set can be partitioned into two sets, $S$ and $C$, so that $S$ induces an independent set $S$ and $C$ incudes a clique. If all edges between $S$ and $C$ are present in $G$, then $G$ will be called the \emph{complete split} graph $CS_{c,s}$, where $c=|C|,s=|S|$. Notice that $CS_{c,s}=\overline{K}_{s} \vee K_{c}$.

%\section{A few useful facts about the signless Laplacian}
Let us now list a few facts about signless Laplacians that will prove useful in the sequel. 

\begin{lem}\cite{Das10}\label{lem:das}
If $\delta(G)>0$, then $\mu(G) < \delta(G)$.
\end{lem}
%A stronger bound than Theorem \ref{thm:das} has been obtained in \cite{Brazilian} but we shall not require it.

\begin{thm}\cite{Brazilian}\label{thm:span}
If $G_{0}$ is a spanning subgraph of $G$, then $$\mu(G_{0}) \leq \mu(G).$$
\end{thm}
%We shall also find it convenient to say that $G$ is a \emph{supergraph} of $G_{0}$.

\begin{thm}[cf. \cite{Brazilian}]\label{thm:cf}
If $G$ is a graph on $n$ vertices, then
%labelled $\{1,\ldots,n\}$, then 
$$
\mu(G)=\min_{x \in \mathbb{R}^{n}-\{0\}}{\frac{x^{T}Qx}{x^{T}x}}=\min_{x \in \mathbb{R}^{n}-\{0\}}{\frac{\sum_{ij \in E(G)}{(x_{i}+x_{j})^{2}}}{x^{T}x}}
$$
\end{thm}
%We shall also find it convenient to say that $G$ is a \emph{supergraph} of $G_{0}$.

\begin{cor}\label{cor:mu4}
Let $H$ be a graph with independent set $S \subseteq V(G)$. Let $T=V(G)-S$ and suppose that the subgraph induced on $T$ has $e$ edges. Then: $$\mu(H) \leq \frac{4e}{|S|+|T|}.$$
\end{cor}
\begin{proof}
Define the vector $x \in \mathbb{R}^{t+s}$ by: 
$$
x_{i}=\begin{cases}
-1  & \text{, if } i \in S \\
1   & \text{, if } i \in T \\
\end{cases}.
$$ 
Now apply Theorem \ref{thm:cf} with the $x$ that we have just defined:
$$
\mu(H) \leq \frac{\sum_{ij \in E(G)}{(x_{i}+x_{j})^{2}}}{x^{T}x}=\frac{4e}{|S|+|T|}.
$$
\end{proof}

\section{Matrix-theoretic tools}\label{sec:tools}

We shall write the eigenvalues of a $n \times n$ Hermitian matrix $M$ in non-decreasing order, \emph{i.e.} $\lambda_{1}(M) \leq \lambda_{2}(M) \ldots \leq \lambda_{n}(M)$. $J_{s,t},J_{t}$ or sometimes simply $J$, will denote the all-ones matrix of a suitable size. The $i$th standard basis (column) vector will be denoted by $e_{i}$. The $i$th row sum of a matrix $A$ will be denoted by $r_{i}(A)$.  
%We shall say that an eigenvector associated  with $\lambda_{1}(M)$ is a \emph{smallest eigenvector} of $M$.

%\begin{defin}
%Let $w \in \mathbb{R}^{n}$. If all the entries of $w$ have the same sign, we say that $w$ is \emph{unisign}.
%\end{defin}

%\begin{defin}
%Let $M$ be a real symmetric matrix whose smallest eigenvalue is simple. If each smallest eigenvector of $M$ is unisign, we shall call $M$ a \emph{minpositive matrix}.
%\end{defin}

\begin{defin}
Let $M$ be a real symmetric matrix whose smallest eigenvalue $\lambda_{1}(M)$ is simple. If $\lambda_{1}(M)$ has a positive eigenvector, then $M$ will be called \emph{minpositive}.
\end{defin}

The class of minpositive matrices is quite wide and includes, for instance, irreducible $Z$-matrices (or more generally, negatives of eventually positive matrices), irreducible inverse-positive (a.k.a. monotone) matrices, and negatives of certain copositive matrices. Note also that since $M$ is Hermitian, minpositivity of $M$ is equivalent to $M^{-1}$ possessing the \emph{strong Perron-Frobenius property} in the sense of \cite{Nou06}.
%We shall use a number of results about matrices, which we record here for convenience.

%We will use the weak version of Weyl's inequalities. For the strong version of Weyl's inequalities see (\cite[Theorem 4.3.7]{HornJohnson}).
The next theorem is the weak version of Weyl's inequalities \cite[Theorem 4.3.1]{HornJohnson}, together with the condition for equality that has been given by Wasin So in \cite{So94}. So's condition is valid also for the strong version of Weyl's inequalities (\cite[Theorem 4.3.7]{HornJohnson}) but the weak version will suffice for our needs here.
\begin{thm}\label{thm:weylso}
Let $A,B$ be Hermitian $n \times n$ matrices. Then for any $k=1,2,\ldots,n$ we have:
%\begin{equation}\label{eq:weyl}
$$
\lambda_{k}(A)+\lambda_{1}(B) \leq \lambda_{k}(A+B) \leq \lambda_{k}(A)+\lambda_{n}(B)
$$
%\end{equation}
For each of these inequalities, equality is satisfied if and only if there exists a nonzero coomon eigenvector $x$ for all three matrices $A,B,A+B$, with the appropriate eigenvalues.
\end{thm}

Let us establish a simple result relating the usual and signless Laplacians. It is slightly reminiscent of \cite[Theorem 2.1]{KirDeb11}.
\begin{thm}\label{thm:mu2}
Let $G$ be a graph with smallest signless Laplacian eigenvalue $\mu$ and largest Laplacian eigenvalue $\lambda$. Let the minimum degree of $G$ be $\delta=\delta(G)$. Then: $$\mu \geq 2\delta-\lambda.$$
Equality obtains if and only if there is a vector $x$ so that:
\begin{itemize}
\item
$Q(G)x=\mu x$,
\item
$L(G)x=\lambda x$,
\item
$D(G)x=\delta x$.
\end{itemize}
\end{thm}
\begin{proof} 
Clearly $Q+L=2D$. By applying Theorem \ref{thm:weylso} for $k=1$ it follows that:
%\begin{equation}\label{eq:qld}
$$
\lambda_{1}(Q+L) \leq \lambda_{1}(Q)+\lambda_{n}(L).
$$
%\end{equation}
But $\lambda_{1}(Q+L)=\lambda_{1}(2D)=2\delta$ and therefore $\mu + \lambda \geq 2 \delta$.
The equality characterization follows from the last part of Theorem \ref{thm:weylso}.
\end{proof}

Recall further that the \emph{Schur complement} of the partitioned matrix  
\begin{equation}\label{eq:M}
M=\left[
        \begin{array}{cc}
           A& B\\
           C& D
        \end{array}
    \right]
\end{equation}
is $M/D=A-BD^{-1}C$, assuming that $D$ is invertible. 
\begin{thm}\cite{Hay68}\label{thm:schur}
Let $M$ be a Hermitian matrix and let $D$ be a nonsingular principal submatrix of $M$. Then $M$ is positive semidefinite if and only if $D$ and $M/D$ are both positive semidefinite.
\end{thm}

\section{Analyzing the $S$-Roth property}\label{sec:basic}
Let $H$ be a connected non-bipartite graph with a maximal independent set $S \subseteq V(G)$ and let $T=V(G)-S$. Two subgraphs of $H$ are of special interest to us: $G_{H,T}$, the subgraph induced by $T$ and $B_{H,S}$, the bipartite subgraph obtained by deleting all the edges between vertices in $T$ from $H$. When in no danger of confusion, we will simply write $G_{H},B_{H}$ or even $G$ or $B$.

Whether $H$ will turn out to be $S$-Roth will depend on the structure of $G_{H,T}$ and $B_{H,S}$ and their interplay. Let us now put the discussion into matrix-theoretic terms. 
Ordering the vertices of $H$ so that the vertices in $T$ are listed first and those in $S$ listed last, we can write the signless Laplacian $Q(H)$ as:
$$
Q(H)=\left[
        \begin{array}{cc}
           Q(G)+D_{1}& K\\
           K^{T}& D_{2}
        \end{array}
    \right],
$$
where $Q(G)$ is the signless Laplacian matrix of $G$ and $D_{1},D_{2}$ are diagonal matrices. We shall henceforth write $Q$ instead of $Q(G)$ when in no danger of confusion. 
%Note also that $K$ is the biadjacency matrix of $B_{H,S}$ (PROVIDE REF).

%%\begin{prop}\label{prop:rows}
%%There are no zero rows in $K$.
%%\end{prop}
%%\begin{proof}
%%A zero row in $K$ corresponds to a vertex in $T$ who has no neighbours in $S$. But the existence of such a vertex would violate the maximality of $S$.
%%\end{proof}

The diagonal entries of $D_{2}$ are simply the degrees of the vertices in $S$. A diagonal entry of $D_{1}$ records the number of vertices in $S$ that are adjacent to the corresponding vertex of $T$.
%how many neighbours in $S$ have the corresponding vertices in $T$. 
%We shall call this number the \emph{$B$-degree} of the vertex, since it is equal to the degree of the vertex when considered as a vertex of $B_{H,S}$. 
%In this parlance, Proposition \ref{prop:rows} says that all $B$-degrees are greater than zero.

Now let $\mu=\mu(H)$ be the smallest eigenvalue of $Q(H)$. Since $H$ is not bipartite, we have $\mu>0$. 

\begin{lem}\label{lem:deg}We have
$\mu<\min_{1 \leq i \leq s}{(D_{2})_{ii}}$
and therefore the matrix $\mu I - D_{2}$ is invertible.
\end{lem}
\begin{proof}
As noted above, the diagonal entries of $D_{2}$ are the degrees of the vertices in $S$. The conclusion follows immediately from Lemma \ref{lem:das}.
\end{proof}

Let $x$ be an eigenvector corresponding to $\mu$ and let us partition $x$ as $x=[w^{T} z^{T}]^{T}$ conformally with the partition of $Q(H)$ (so we have $w=x(T),z=x(S)$). We now write down the eigenequation:
\begin{equation}\label{eq:x}
\left[
        \begin{array}{cc}
           Q+D_{1}& K\\
           K^{T}& D_{2}
        \end{array}
\right] 
\left[
        \begin{array}{c}
           w\\
           z
        \end{array}
\right]= \mu
\left[
        \begin{array}{c}
           w\\
           z
        \end{array}
\right].
\end{equation}
%We remark that $w=x(S),z=x(T)$.
Multiplying out we have the following system of equations:
\begin{equation}\label{eq:sys1}
(Q+D_{1})w+Kz=\mu w
\end{equation}
\begin{equation}\label{eq:sys2}
K^{T}w+D_{2}z=\mu z
\end{equation}

Lemma \ref{lem:deg} ensures that equation (\ref{eq:sys2}) can be solved as:
\begin{equation}\label{eq:z}
z=(\mu I-D_{2})^{-1}K^{T}w
\end{equation}

We can now give a useful characterization of $S$-Rothness.

\begin{prop}\label{prop:roth}
The graph $H$ is $S$-Roth if and only if for every eigenvector $x$ corresponding to $\mu(H)$ it holds that $x(S)>0$ or $x(S)<0$.
%and partitioned as $x=[w^{T} z^{T}]^{T}$ it holds that $w>0$ or $w<0$.
\end{prop}
\begin{proof}
One direction is trivial from the definition of an $S$-Roth graph. For the other direction, consider a partition of $x$ as in (\ref{eq:x}), with $w=x(S)$ and $z=x(T)$.
%$x=[w^{T} z^{T}]^{T}$. 
Since $H$ is connected, every vertex in $S$ must have at least one neighbour in $T$. This means that $K^{T}$ has no zero row. Furthermore, the matrix $\mu I-D_{2}$ is diagonal and all its diagonal entries are negative, by Lemma \ref{lem:deg}. 
Therefore, if $w>0$ it follows from (\ref{eq:z}) that $z<0$ and we are done.
\end{proof}

Let us now substitute the expression for $z$ found in (\ref{eq:z}) into (\ref{eq:sys1}):
$$(Q+D_{1})w+K(\mu I-D_{2})^{-1}K^{T}w=\mu w$$

In other words, $(\mu,w)$ is an eigenpair of the following matrix:
\begin{equation}\label{eq:qmu}
Q_{\mu}=Q+D_{1}+K(\mu I-D_{2})^{-1}K^{T}.
\end{equation}

\begin{lem}\label{lem:smallest}
$\mu$ is the smallest eigenvalue of $Q_{\mu}$.
\end{lem}
\begin{proof}
Consider the shifted matrix $Q(H)-\mu I$:
$$
Q(H)-\mu I=\left[
        \begin{array}{cc}
           Q+D_{1}-\mu I& K\\
           K^{T}& D_{2}- \mu I
        \end{array}
    \right],
$$
Lemma \ref{lem:deg} enables us to take the Schur complement of $Q(H)-\mu I$ by its bottom-right block:
$$(Q(H)-\mu I)/(D_{2}- \mu I)=Q+D_{1}-\mu I+K(\mu I-D_{2})^{-1}K^{T}=Q_{\mu}- \mu I$$

The matrix $Q(H)-\mu I$ is clearly positive semidefinite and thus by Theorem \ref{thm:schur} the matrix $Q_{\mu}- \mu I$ is positive semidefinite as well. This means that all eigenvalues of $Q_{\mu}$ are greater than or equal to $\mu$.
\end{proof}

We have thus arrived at an alternative characterization of the $S$-Roth property.

\begin{thm}\label{thm:qmu}
$H$ is $S$-Roth if and only if $Q_{\mu}$ is a minpositive matrix.
\end{thm}
\begin{proof}
Immediate from Lemma \ref{lem:smallest} and Proposition \ref{prop:roth}.
\end{proof}

%%In Sections \ref{sec:gamma}, \ref{sec:kst} and \ref{sec:kst_sparse} we shall use Theorem \ref{thm:qmu} to obtain many classes of Roth pairs. Most times we shall show that $Q_{\mu}$ is an irreducible $Z$-matrix. An exception is Theorem \ref{thm:Gdeg2} in whose proof we shall show that $Q_{\mu}$ is minpositive in spite of having offdiagonal entries with different signs.

%%%\begin{rmrk}
%%%A word of caution. The matrix $Q_{\mu}$ may have a positive eigenvector associated not with $\mu$ but with some other eigenvalue. In that case, $Q_{\mu}$ is \emph{not} minpositive. Examples of this phenomenon are easy to come by in the case that $G$ is $d$-regular. 
%%%%Then the row sums of $Q_{\mu}$ all equal $2d+s-\alpha t$ and this is an eigenvalue corresponding to $\mathfrak{1}_{t}$. However, it might well happen that $2d+s-\alpha t > \mu$. For an example, take $t=10,s=3,d=4$.
%%%\end{rmrk}

\section{A combinatorial sufficient condition for $S$-Rothness}\label{sec:gamma}
%Theorem \ref{thm:qmu} provides us with with the means to bring some powerful tools from matrix theory to bear on the quesiton of $S$-Rothness. In this section we wish to employ the theorem in order to find a simple combinatorial condition for the $S$-Rothness of $H$, that will depend on $B=B_{H,S}$ and $G=G_{H,T}$.

We would now like to formulate simple combinatorial conditions on $B_{H,S}$ and $G_{H,T}$ that will ensure that the graph $H$ is $S$-Roth, drawing upon Theorem \ref{thm:qmu}. We continue to assume in this section that the vertices of $H$ are sorted so that those of $T$ come first, then those of $S$. Note also that by convention an empty sum equals zero.

\begin{prop}\label{prop:quadform}
Let $i,j$ be distinct indices in $\{1,2,\ldots,t\}$ and let $N_{i,j}=\{k \in S|k \sim_{B} i, k \sim_{B} j\}$ be the set of their common neighbours in $S$. Then:
$$(K(D_{2}-\mu I)^{-1}K^{T})_{ij}=\sum_{k \in N_{ij}}{\frac{1}{d_{B}(k)-\mu}}.$$
\end{prop}
\begin{proof}
%Immediate, considering that the diagonal entries of $D_{2}$ are the numbers $d_{1},d_{2},\ldots,d_{s}$.
Immediate, upon observing that the diagonal entries of $D_{2}$ are the degrees of the vertices in $S$. 
\end{proof}

We can now state the main result of this section. Recall that a real matrix is called a $Z$-matrix if all of its offdiagonal entries are nonpositive. 

\begin{thm}\label{thm:harmcond}
Let $H$ be a connected non-bipartite graph with a maximal independent set $S \subseteq V(H)$. Suppose that
\begin{itemize}
\item
For all $ij \in E(G)$: $$\sum_{k \in N_{ij}}{\frac{1}{d_{B}(k)}} \geq 1$$ 
and that
\item 
For all $ij \notin E(G)$: $$N_{ij} \neq \emptyset.$$
\end{itemize}
Then $H$ is $S$-Roth.
\end{thm}
\begin{proof}
Consider the $ij$th off-diagonal entry of $Q_{\mu}$: by Equation (\ref{eq:qmu}) and Proposition \ref{prop:quadform} it is equal to $1-\sum_{k \in N_{ij}}{\frac{1}{d_{B}(k)-\mu}}$ if $ij \in E(G)$ and to $-\sum_{k \in N_{ij}}{\frac{1}{d_{B}(k)-\mu}}$ if $ij \notin E(G)$. Therefore, by our assumptions (and the fact that $\mu>0$), it is negative in both cases. Thus $Q_{\mu}$ is a $Z$-matrix all of whose off-diagonal entries are strictly negative, ergo it is minpositive. We are done by Theorem \ref{thm:qmu}.
\end{proof}

\begin{cor}\label{cor:gc}
Let $c^{B}=\max_{k \in S}{d_{B}(k)}$ and suppose that $|N_{ij}| \geq c^{B}$ for all $ij \in E(G)$ and that $N_{ij} \neq \emptyset$ for all distinct $ij \notin E(G)$. Then $H$ is $S$-Roth.
\end{cor}
\begin{proof}
The condition $|N_{ij}| \geq c^{B}$ clearly implies that $\sum_{k \in N_{ij}}{\frac{1}{d_{B}(k)}} \geq |N_{ij}|\frac{1}{c^{B}} \geq 1$.
\end{proof}

\begin{cor}\label{cor:bdeg}
%If the $B$-degree of every vertex in $T$ is 
If $d_{B}(i) \geq \frac{t+s}{2}$ for every vertex $i$ in $T$, then $H$ is $S$-Roth.
\end{cor}
\begin{proof}
A simple counting argument shows that in this case $|N_{ij}| \geq t$ for all $i,j$. On the other hand, clearly $c^{B} \leq t$.
\end{proof}

%Note that if $s \geq t$, then $\frac{t+s}{2} \geq s$ and thus Corollary \ref{cor:bdeg} is implies Corollary \ref{cor:st}.

\begin{cor}\label{cor:st}
Supppose that $B=K_{s,t}$. If $s \geq t$, then $H$ is $S$-Roth.
\end{cor}
\begin{proof}
In this case $|N_{ij}|=s$ for all $i,j$ and $c^{B}=t$.
\end{proof}

\begin{rmrk}
The difference between Theorem \ref{thm:harmcond} and Corollary \ref{cor:gc} is that the theorem posits a more refined ``local'' condition, whereas the corollary operates via a cruder ``global'' condition.
\end{rmrk}

%Our focus in this section was mainly on $B_{H,S}$. We now fix $B_{H,S}=K_{s,t}$, as in Corollary \ref{cor:st} and turn to study in greater detail the effect of $G_{H,T}$ on the $S$-Rothness of $G$. If $s \geq t$ we already know that $H$ is $S$-Roth. Some results about the case $s<t$ will be presented in the next section.

Note that when $B=K_{s,t}$ it does not matter in which way $G$ is ``glued'' to $B$, and $S$-Rothness depends only on $G$ in itself. This may not always true for other fixed graphs $B$, of course. See Remark \#1 of Section \ref{sec:open} for a further discussion of this issue.

\section{An example}\label{sec:yeast}
The graph in Figure \ref{fig:ex1} appears in \cite{Protein} as an example of a nearly-complete almost-bipartite subgraph found in the Uetz network by the algorithm proposed in \cite{Protein}.

\begin{figure}[here]
\includegraphics[height=8cm,width=10cm]{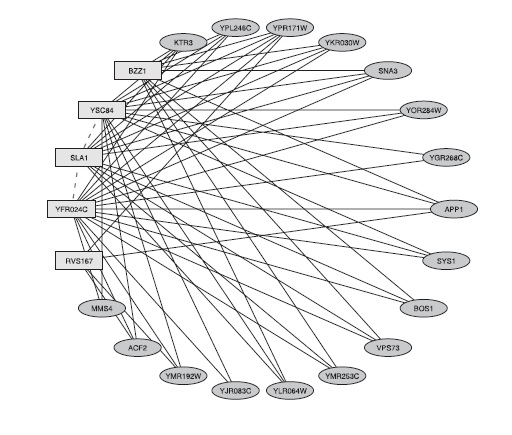}
\caption{A protein graph - drawing reproduced from \cite[Figure 6(a)]{Protein}}\label{fig:ex1}
\end{figure}

This graph has a total of $22$ vertices, and as Figure \ref{fig:ex1} makes abundantly clear we can view it as a noisy version of a $K_{5,17}$. We will now show that Corollary \ref{cor:gc} predicts that the smallest eigenvector of $Q(H)$ will be positive on the five ``rectangular" vertices and negative on the other seventeen ``elliptical" vertices. Indeed, we have $s=17,t=5$ and $c^{B}=5$. Since the graph $G$ in this case has only two edges (SLA1 $\sim$ YSC84 and SLA1 $\sim$ YFR024C) we have to check $N_{ij}$ only for them. Visual inspection shows that these two pairs have $8$ and $10$ common neighbours in $S$, respectively. Therefore the condition of Corollary \ref{cor:gc} is met.

Our example shows that the examination of the smallest signless Laplacian eigenvector suffices to correctly identify the underlying structure present in the graph. Of course, this is just one example and others should be examined in the future to establish the practical applicability of the method.

\section{A case study}\label{sec:case}
In order to better understand Theorem \ref{thm:harmcond} and its scope of applicability we examined three special cases $s \in \{5,7,9\},t=4$ with $G=K_{4}$ for all possible \emph{connected} bipartite graphs $B$. This has been made possible by using data made publicly available by Gordon Royle \cite{RoyleBipData}.

Each row of Table \ref{tab:royle} records information about one value of $s$. The six columns of the table enumerate the following: the value of $s$, the total number of connected bipartite graphs, the number of $S$-Roth graphs amongst them, the number of graphs that satisfy the conditions of Theorem \ref{thm:harmcond}, the number of graphs such that their $Q_{\mu}$ is an $M$-matrix, and the number of graphs such that $Q_{\mu}^{-1}$ is entrywise positive.

\begin{table}[here]
\caption{$S$-Rothness for small bipartite graphs}
\centering 
\begin{tabular}{c c c c c c} 
\hline\hline 
s\ & \# graphs\ & $S$-Roth\ & Theorem \ref{thm:harmcond}\ & $Q_{\mu}$ is an $M$-matrix & $Q_{\mu}^{-1}>0$\ \\ [0.5ex] 
\hline 
5 & 558 & 64 & 4 & 23 & 35 \\
7 & 5375 & 823 & 85 & 283 & 515 \\
9 & 36677 & 8403 & 1234 & 3155 & 6054 \\
\hline 
\end{tabular}
\label{tab:royle} 
\end{table}

Some observations from Table \ref{tab:royle}: Theorem \ref{thm:harmcond} becomes more powerful as $s$ increases with respect to $t$. On the other hand, we see that there are many cases when $Q_{\mu}$ is an $M$-matrix which are not accounted for by Theorem \ref{thm:harmcond}; a possible way to extend its coverage would be by incorporating into the argument some lower bounds on $\mu(H)$. These are, however, notoriously hard to come by.

It is worthwhile to point out thay by choosing $G$ to be a clique we are, so to say, taxing to the utmost Theorem \ref{thm:harmcond}. Note that our concepts make sense for disconnected $B$ as well but we have chosen to omit them from this study, hopefully with little loss. 

Let us now consider some examples to illustrate the possibilities.
All examples will be  drawn from the case $s=7,t=4,G=K_{4}$.
\begin{expl}
Let $B$ be graph \# 4530 in Royle's catalogue. In this case the matrix $K$ is given by:

$$K=
\left(\begin{array}{ccccccc}
1&1&1&1&0&0&0\\
1&1&1&1&0&0&0\\
1&1&1&1&0&0&0\\
1&1&1&1&1&1&1\\
\end{array}
\right).
$$

The degrees $d_{B}(1),d_{B}(2),\ldots,d_{B}(7)$ are  $4,4,4,4,1,1,1$ (read off as the column sums of $K$). The condition of Theorem \ref{thm:harmcond} is met as $N_{ij}=\{1,2,3,4\}$ for all $i,j$ and therefore $\sum_{k \in N_{ij}}{\frac{1}{d_{B}(k)}}=\frac{1}{4}+\frac{1}{4}+\frac{1}{4}+\frac{1}{4}=1$.

And indeed, the matrix $Q_{\mu}$ can be computed as:

$$
Q_{\mu}=\left(\begin{array}{cccc}
5.8123&-0.18774&-0.18774&-0.18774\\
-0.18774&5.8123&-0.18774&-0.18774\\
-0.18774&-0.18774&5.8123&-0.18774\\
-0.18774&-0.18774&-0.18774&0.65427\\
\end{array}
\right), \mu=0.63226.
$$

The eigenvector $x$ is:

$$
x=
\left(\begin{array}{ccccccccccc}
0.008\\
0.008\\
0.008\\
0.2057\\
-0.0682\\
-0.0682\\
-0.0682\\
-0.0682\\
-0.5594\\
-0.5594\\
-0.5594\\
\end{array}
\right).
$$
\end{expl}

\begin{expl}
Now consider an example where $Q_{\mu}$ is an $M$-matrix but the condition of Theorem \ref{thm:harmcond} is not met. Let $G$ be graph \# 5104 in Royle's catalogue. We have

$$
K=
\left(\begin{array}{ccccccc}
1&1&1&1&1&0&0\\
1&1&1&0&0&1&1\\
1&1&0&1&1&1&0\\
1&1&1&1&0&1&0\\
\end{array}
\right), 
$$
and
$$
[d_{B}(1),d_{B}(2),\ldots,d_{B}(7)]=[4,4,3,3,2,3,1].
$$

We observe that $N_{12}=\{1,2,3\}$ and therefore $\sum_{k \in N_{12}}{\frac{1}{d_{B}(k)}}=\frac{1}{4}+\frac{1}{4}+\frac{1}{3}=\frac{5}{6}<1$. However, 

$$
Q_{\mu}=\left(\begin{array}{cccc}
5.6058&-0.08776&-0.93542&-0.54653\\
-0.08776&0.88934&-0.08776&-0.54653\\
-0.93542&-0.08776&5.6058&-0.54653\\
-0.54653&-0.54653&-0.54653&5.9947\\
\end{array}
\right)
$$

is an $M$-matrix ($\mu=0.82028$).

\end{expl}

\begin{expl}
Consider now $B$ which is graph \# 3503 in Royle's catalogue.

$$
K=
\left(\begin{array}{ccccccc}
1&1&1&0&0&0&0\\
1&0&0&1&1&0&0\\
1&0&0&0&0&1&1\\
1&1&1&1&1&1&1\\\end{array}
\right), 
$$
$$
[d_{B}(1),d_{B}(2),\ldots,d_{B}(7)]=[4,2,2,2,2,2,2].
$$

Now $\sum_{k \in N_{ij}}{\frac{1}{d_{B}(k)}}=\frac{1}{4}$  for distinct $i,j$ in $\{1,2,3\}$ and $\sum_{k \in N_{i4}}{\frac{1}{d_{B}(k)}}=1 \frac{1}{4}$ for $i=1,2,3$. Although $Q_{\mu}$ has many positive entries, its inverse is nevertheless positive and therefore $Q_{\mu}$ is minpositive. In this case $\mu=1.0922$. In fact,

$$
Q_{\mu}=\left(\begin{array}{cccc}
3.453&0.6561&0.6561&-1.547\\
0.6561&3.453&0.6561&-1.547\\
0.6561&0.6561&3.453&-1.547\\
-1.547&-1.547&-1.547&3.0468\\
\end{array}
\right),
$$

$$
Q_{\mu}^{-1}=\left(\begin{array}{cccc}
0.37674&0.019201&0.019201&0.21078\\
0.019201&0.37674&0.019201&0.21078\\
0.019201&0.019201&0.37674&0.21078\\
0.21078&0.21078&0.21078&0.64927\\
\end{array}
\right),
$$

%%$$
%%w=
%%\left(\begin{array}{cccc}
%%0.10263\\
%%0.10263\\
%%0.10263\\
%%0.24367\\
%%\end{array}
%%\right).
%%$$

\end{expl}

\begin{expl}
$B$ is now graph \# 1447 in Royle's catalogue. Then

$$
K=
\left(\begin{array}{ccccccc}
1&1&0&0&0&0&0\\
0&0&1&1&0&0&0\\
1&0&1&0&1&0&0\\
1&1&1&1&1&1&1\\
\end{array}
\right), 
$$
and
$$
[d_{B}(1),d_{B}(2),\ldots,d_{B}(7)]=[3,2,3,2,2,1,1].
$$

%The computation of $N_{ij}$ can by now be safely left to the reader who no doubt has become proficient in it. 
Since $N_{12}=\emptyset$ we cannot even try to apply Theorem \ref{thm:harmcond}. Also, $Q_{\mu}$ is not monotone. Nevertheless, it can be found by computation that $w>0$. Indeed, we have in this case:
$$
Q_{\mu}=\left(\begin{array}{cccc}
3.8172&1&0.57038&-0.18282\\
1&3.8172&0.57038&-0.18282\\
0.57038&0.57038&4.3876&-0.61244\\
-0.18282&-0.18282&-0.61244&0.77727\\
\end{array}
\right),
$$

$$
w=
\left(\begin{array}{cccc}
0.0047565\\
0.0047565\\
0.033593\\
0.21264\\
\end{array}
\right),  \mu=0.67234.
$$

We remark that $Q_{\mu}^{-6}>0$ in this case but there seems to be no easy combinatorial interpretation of this fact.

\end{expl}

\section{The graphs $H=\overline{K}_{s} \vee G$ for $s<t$, Part I}\label{sec:kst}

Note that if $B_{H,S}=K_{s,t}$, then $H=\overline{K}_{s} \vee G$. From this point on we shall consider only graphs of this form.  When we say that $H$ is $S$-Roth we shall invariably mean that $S$ is the set of vertices inducing the $\overline{K}_{s}$.

Since $B=K_{s,t}$, we have now $K=J_{t,s},D_{1}=sI_{t}$, and $D_{2}=tI_{s}$. Therefore the definition of the matrix $Q_{\mu}$ in Equation (\ref{eq:qmu}) simplifies to:
\begin{equation}\label{eq:qmu_kst}
Q_{\mu}=Q+sI-\frac{s}{t-\mu}J.
\end{equation}

\begin{defin}
Let $H=\overline{K}_{s} \vee G$. Let $\mu=\mu(H)$ be the smallest signless Laplacian eigenvalue of $H$. Define the quantity $\alpha_{H}(G)$ as:
$$\alpha_{H}(G)=\frac{s}{t-\mu}.$$
\end{defin}

%The assumptions that $t<s$ and that $B_{H,S}=K_{s,t}$ will be in force for the rest of this section.

%Whenever $\alpha_{H}(G)>1$, $Q_{\mu}$ is a $Z$-matrix whose oofdiagonal entries are all strictly negative, and thus minpositive. We can state this fact as:
\begin{lem}\label{lem:alpha}
If $t>s$ and $\alpha_{H}(G)>1$, then $H$ is $S$-Roth.
\end{lem}
\begin{proof}
If $\alpha_{H}(G)>1$, then by (\ref{eq:qmu_kst}) $Q_{\mu}$ is a $Z$-matrix all of whose offdiagonal entries are strictly negative. Therefore it is minpositive and the conclusion follows immediately from Theorem \ref{thm:qmu}.
\end{proof}

\begin{rmrk}\label{rem:mu}
Notice that $\alpha_{H}(G)=1$ is equivalent to $\mu(H)=t-s$ and $\alpha_{H}(G)>1$ is equivalent to $\mu(H)>t-s$.
\end{rmrk}

%We need some facts about the usual Laplacian matrices of joins.

\begin{thm}\label{thm:Gdeg}
Let $H=\overline{K}_{s} \vee G$, with $t>s$. Suppose that one of the following cases holds: 
\begin{itemize}
\item
(A) $\delta(G)>t-s$; OR 
\item
(B) $\delta(G)=t-s$ and $\overline{G}$ is connected. 
\end{itemize}

\end{thm}
Then $H$ is $S$-Roth.
\begin{proof}
The degrees of the vertices in $S$ all equal $t$. On the other hand, the degrees of the vertices in $T$ are all at least $s+\delta(G)$. Since we assumed $\delta(G) \geq t-s$ in both cases, we have that $\delta(H) \geq t$ in both of them, with equality obtaining in case (B).

Now by Lemma \ref{lem:join} and Theorem \ref{thm:mu2} we have that 
\begin{equation}\label{eq:qld}
%$$
\mu \geq 2t-(t+s)=t-s.
%$$
\end{equation}
Therefore, $\alpha_{H}(G) \geq 1$ and we see that $Q_{\mu}$ is a $Z$-matrix. 

In case (A) we have $\alpha_{H}(G)>1$ and so we are done by Lemma \ref{lem:alpha}. Otherwise, $\alpha_{H}(G)=1$ and $Q_{\mu}=Q+sI-J$ is a $Z$-matrix whose offdiagonal zero pattern is exactly the same as that of the adjacency matrix of $\overline{G}$. Therefore, if $\overline{G}$ is connected, then $Q_{\mu}$ is irreducible and thus a minpositive matrix.  
\end{proof}

%To illustrate that $\delta(G)=t-s$ alone is not sufficient to ensure that $H$ is $S$-Roth, we have the following example.
If $\delta(G)=t-s$ and $\overline{G}$ is disconnected, then $H$ may fail to be $S$-Roth, as Example \ref{expl:42} will illustrate. 

\begin{expl}\label{expl:42}
Let $t=6,s=4, G=K_{4,2}$. In this case, we have $\delta(G) =t-s=2$, but the smallest eigenvector of $Q(H)$ has two zero entries. 

Indeed, in this case we have:
$$
Q_{\mu}=\left(\begin{array}{cccccc} 5 & -1 & -1 & -1 & 0 & 0\\ -1 & 5 & -1 & -1 & 0 & 0\\ -1 & -1 & 5 & -1 & 0 & 0\\ -1 & -1 & -1 & 5 & 0 & 0\\ 0 & 0 & 0 & 0 & 7 & -1\\ 0 & 0 & 0 & 0 & -1 & 7 \end{array}\right)
$$ and $\mu=t-s=2,w=[1,1,1,1,0,0]^{T}$.
\end{expl}

Our next result supplies a complete characterization of $S$-Rothness when $\delta(G)=t-s$ and $\overline{G}$ is disconnected:
\begin{thm}\label{thm:steve}
Let $H=\overline{K}_{s} \vee G$ with $t>s$ and suppose that $\delta(G)=t-s$, that $\overline{G}$ is disconnected, and that $G=G_{1} \vee G_{2} \vee \ldots \vee G_{k}$ is a maximal join decomposition.
Then $H$ is $S$-Roth if and only if for each $j=1,2,\ldots,k$ there is a vertex $v_{j} \in V(G_{j})$ such that $d_{G}(v_{j})>t-s$.
\end{thm}
\begin{proof}
Suppose first without loss of generality that for every vertex $v \in V(G_{1})$ it holds that $d_{G}(v)=t-s$. Furthermore, $d_{G_{1}}(v)$ is equal to $d_{G}(v)$ minus the number of edges between $v$ and the subgraph of $G$ induced by $G_{2} \vee \ldots \vee G_{k}$. That is: 
\begin{equation}\label{eq:dg1}
d_{G_{1}}(v)=(t-s)-(t-|V(G_{1})|)=|V(G_{1})|-s.
\end{equation}

Define now a vector $x \in \mathbb{R}^{|V(H)|}$ by: 
$$
x_{i}=\begin{cases}
-|V(G_{1})|  & \text{, if } i \in S \\
s   & \text{, if } i \in V(G_{1}) \\
0   & \text{, otherwise.} \\
\end{cases}
$$
Using \eqref{eq:dg1} it is straightforward to verify that $Q(H)x=(t-s)x$. Therefore, $x$ is an eigenvector of $Q(H)$ corresponding to the eigenvalue $t-s$. From the proof of Theorem \ref{thm:Gdeg} we know that $\mu \geq t-s$ and therefore $\mu=t-s$. But since $x$ has zero entries, we deduce that $H$ is not $S$-Roth.

Suppose now that there is a vertex $v_{j} \in V(G_{j})$ such that $d_{G}(v_{j})>t-s$ for every $j$. As in the proof of Theorem \ref{thm:Gdeg} we have $\mu \geq t-s$. If $\mu>t-s$ then $H$ is $S$-Roth by Lemma \ref{lem:alpha}. We are now going to show that the case $\mu=t-s$ is in fact impossible. 
%We can also assume that $\mu=t-s$ since otherwise we would be done by Lemma \ref{lem:alpha}. 

Let $x$ be an eigenvector of $Q(H)$ corresponding to $\mu$. Since $\mu=t-s$ we see that equality holds throughout \eqref{eq:qld} and therefore we infer from Theorem \ref{thm:mu2} that $(x,2t)$ is an eigenpair of $2D(H)$ and that $(x,t+s)$ is an eigenpair of $L(H)$. Now, the first statement immediately implies that $x$ is nonzero only on vertices of degree $t$ in $H$, whereas the second statement implies, via Lemma \ref{lem:join_max}, that $x$ is constant over $S$ and over $V(G_{i})$ for every $i$. Put together with our assumption, these two observations imply that $x=0$. This is a contradiction, and therefore the case $\mu=t-s$ is impossible. 
%Hence, $\mu>t-s$ and $H$ is $S$-Roth by Lemma \ref{lem:alpha}.
\end{proof}

We can deduce a rudimentary extremal-type result:
\begin{cor}\label{cor:extremal}
$H=\overline{K}_{s} \vee G$, with $2 \leq s<t$. If $G$ has at least $\binom{t}{2}-(s-2)$ edges, then $H$ is $S$-Roth.
\end{cor}
\begin{proof}
We can view $G$ as $K_{t}$ from which a number of edges have been deleted. 
Even after the deletion of $s-2$ edges, the degree of any vertex is at least $(t-1)-(s-2)=t-s+1$ and so we are done by Theorem \ref{thm:Gdeg}.
\end{proof}

\section{Interlude - A new analysis of $S$-Rothness}\label{sec:q_alt}
We shall now revisit the analysis of the 
eigenequations (\ref{eq:sys1}) and (\ref{eq:sys2}), with the vectors $w$ and $z$ trading roles this time. We only consider the case $B=K_{s,t}$ so we may at once take $K=J_{t,s}$ and $D_{1}=sI,D_{2}=tI$. 

Now we introduce a new matrix:
$$R_{\mu}=(Q+sI-\mu I).$$
We are going to assume that $R_{\mu}$ is positive definite (this holds very often - whenever $s>\mu$ and in some other cases as well).
%We will make a crucial assumption here: that $s>\mu$. 
%In this case $R_{\mu}$ is positive definite.

Consider first Equation (\ref{eq:sys1}).
%If we assume that the matrix $(Q+sI-\mu I)$ is nonsingular, 
We can solve it as:
\begin{equation}\label{eq:new_w}
w=-R_{\mu}^{-1}Jz.
\end{equation}
Substituting (\ref{eq:new_w}) into (\ref{eq:sys2}) we obtain:
$$
-J^{T}R_{\mu}^{-1}Jz+tz=\mu z.
$$
The matrix $J^{T}R_{\mu}^{-1}J$ is clearly a multiple of $J$, namely $\gamma J$, where $\gamma$ is the sum of all entries in $R_{\mu}^{-1}$. Note that $\gamma>0$, since $R_{\mu}^{-1}$ is positive definite. Therefore we can write:
\begin{equation}\label{eq:new_z}
\gamma Jz=(t-\mu)z. 
\end{equation}
Observe now that $Jz=\sigma_{z}\mathbf{1}$ where $\sigma_{z}$ is the sum of all entries in $z$. It is impossible to have $\sigma_{z}=0$ since it would imply $z=0$ by (\ref{eq:new_z}) and consequently also $w=0$ by (\ref{eq:new_w}) - a contradiction. Therefore we infer that
\begin{equation}\label{eq:z_val}
z=\frac{\gamma \sigma_{z}}{t-\mu}\mathbf{1}.
\end{equation}

Now assume without loss of generality that $z<0$; we consult again Equation (\ref{eq:new_w}) and see that $w$ is a positive multiple of the vector of row sums of $R_{\mu}^{-1}$. We can summarize our findings as follows:

\begin{thm}\label{thm:rmu}
Let $H=\overline{K}_{s} \vee G$. Suppose that $R_{\mu}$ is positive definite. Then $H$ is $S$-Roth if and only if the row sums of $R_{\mu}^{-1}$ are all positive.
\end{thm}

\begin{rmrk}
If $\mu(H)=s$ and $G$ is bipartite, then $R_{\mu}$ is singular. For example, this happens when $s=3,t=14,G=C_{14}$. In that example we have $z=0$ and this $H$ is not $S$-Roth.
\end{rmrk}

%\begin{rmrk}
%By taking the sums of both sides of (\ref{eq:z_val}) we can show that $\gamma=\alpha_{H}^{-1}$.
%\end{rmrk}

\section{Some more matrix tools}\label{sec:more}

The preceding section saw the injection of inverse matrices into the discussion. Their analysis will require two more tools which we present here. The matrices we shall deal with in the next section will be of the form $Q(C_{k})+\lambda I$ and are easily seen to be strictly diagonally dominant.

The study of strictly diagonally dominant matrices is an old and venerable enterprise. The fact, crucially useful to us here, that a weaker form of diagonal dominance carries over to the inverse seems to have been noticed first by Ostrowski \cite{Ost52} in 1952. We will use a slightly more recent result that quantifies this statement:
\begin{thm}\cite[Theorem 2.4]{LiHuaSheLi07}\label{thm:diagdom}
Let $A$ be a strictly diagonally dominant matrix. Let $\widetilde{A}=A^{-1}=(\widetilde{a}_{ij})$. Then we have:
$$
|\widetilde{a}_{ji}| \leq \max_{l \neq i} \left\{ \frac{|a_{li}|}{|a_{ll}|-\sum_{k \neq l,i}|a_{lk|}} \right\} |\widetilde{a}_{ii}|, \quad \textit{for all} \quad j \neq i.
$$
\end{thm}

We will also use a remarkable result of Bai and Golub, that also gives us information about $A^{-1}$.

\begin{thm}\cite{BaiGol97}\label{thm:bai}
Let $A \in \mathbb{R}^{n,n}$ be a symmetric positive definite matrix whose eigenvalues lie in $[a,b]$, with $a>0$. Furthermore, let $m_{1}=\Tr{A},m_{2}=||A||_{F}^{2}$. Then:
$$
\left[
        \begin{array}{cc}
           m_{1}& n          
        \end{array}
    \right]
    \left[
        \begin{array}{cc}
           m_{2}& m_{1}\\
           b^{2} & b         
        \end{array}
    \right]^{-1}
    \left[
        \begin{array}{c}
           n \\
           1          
        \end{array}
    \right] \leq \Tr{A^{-1}}
$$
and
$$    
\Tr{A^{-1}} \leq
    \left[
        \begin{array}{cc}
           m_{1}& n          
        \end{array}
    \right]
    \left[
        \begin{array}{cc}
           m_{2}& m_{1}\\
           a^{2} & a        
        \end{array}
    \right]^{-1}
    \left[
        \begin{array}{c}
           n \\
           1          
        \end{array}
    \right].
$$
\end{thm}

Now let $A=Q(C_{k})+\lambda I$, with $\lambda>0$. We immediately
observe that $A$ is both symmetric positive definite and strictly diagonally dominant. Notice also that $A$ is a circulant matrix and therefore $A^{-1}$ is also circulant (cf. \cite[p. 74]{Circulant}) and therefore all diagonal entries of $A^{-1}$ are equal to $\Tr{A^{-1}}/k$. Combining Theorem \ref{thm:diagdom} with the upper bound of Theorem \ref{thm:bai}, we have the following bounds on the entries of $A$:
\begin{lem}\label{lem:bai}
Let $A=Q(C_{k})+\lambda I,\lambda>0$ and let $\widetilde{A}=A^{-1}$. Then:

%\begin{itemize}
%\item
$$\widetilde{a}_{ii} \leq \frac{\lambda+1}{\lambda(\lambda+3)}$$ 
and 
$$|\widetilde{a}_{ij}| \leq \frac{1}{\lambda+1} |\widetilde{a}_{ii}|, \forall j \neq i.$$
%\item
%\end{itemize}

\end{lem}
\begin{proof}
The first inequality follows from Theorem \ref{thm:bai}, upon taking $a=\lambda,m_{1}=k(2+\lambda),m_{2}=k(2+\lambda)^{2}+2k$ and doing some algebra. The second inequality follows from Theorem \ref{thm:diagdom} quite easily, upon careful index-chasing.
\end{proof}

\section{The graphs $H=\overline{K}_{s} \vee G$ for $s<t$, Part 
II}\label{sec:cycle}

In Section \ref{sec:kst} we studied the situation when $G$ is relatively dense. Now we turn to investigate the case when $G$ is sparse. It will be seen that the matrix $R_{\mu}$ provides a handier tool than $Q_{\mu}$ in this case, because Theorem \ref{thm:rmu} requires us to check a simpler property than Theorem \ref{thm:qmu} - provided that the inverse $R_{\mu}^{-1}$ is sufficiently understood. Fortunately, if $G$ is sparse then the structure of $R_{\mu}^{-1}$ can often be inferred from that of $R_{\mu}$.

The main result of this section is:
\begin{thm}\label{thm:deg2}
Let $t>s \geq 6$ and let $H=\overline{K}_{s} \vee G$. If $\Delta(G) \leq 2$, then $H$ is $S$-Roth.
\end{thm}
\begin{proof}
Obviously, $G$ is a disjoint union of cycles, paths and isolated vertices. Therefore it has at most $t$ edges. By Corollary \ref{cor:mu4}, we see that $\mu(H) \leq \frac{4t}{t+s}<4$. Thus
we may use Theorem \ref{thm:rmu}. The matrix $R_{\mu}$ is block-diagonal, each block corresponding to a connected component of $G$. We can write $$R_{\mu}=A_{1} \oplus \ldots \oplus A_{k} \oplus B_{1} \oplus \ldots \oplus B_{m} \oplus (s-\mu)I,$$ with the $A_{i}$s corresponding to the cycles, the $B_{j}$s to the paths and the last summand lumping together all isolated vertices, if there are any.
Clearly:
$$R_{\mu}^{-1}=A_{1}^{-1} \oplus \ldots \oplus A_{k}^{-1} \oplus B_{1}^{-1} \oplus \ldots \oplus B_{m}^{-1} \oplus (s-\mu)^{-1}I,$$
and so we need to show that the row sums of each $A_{i}^{-1}$ and each $B_{j}^{-1}$ are positive. 

The row sums of $A_{i}$ are all equal to $4+s-\mu$ and therefore the row sums of $A_{i}^{-1}$ are all equal to $(4+s-\mu)^{-1}$ and thus are positive. Denote for further use $\beta=(4+s-\mu)^{-1}$.

Finally, let us consider a $B_{j}$. It corresponds to a path component of $G$ on, say, $k$ vertices. Had this component been a cycle its matrix $A$ would have had positive row sums, by the preceding argument.
%Had $G$ a cycle component on $k$ vertices, then by the preceding argument there would have been a matrix $C$ corresponding to it and having positive row sums. 
We shall want to write $B_{j}$ as a rank-one modification of $A$ and to show that the row sums remain positive:
$$
B_{j}=A-E, \quad E=(e_{1}+e_{k})(e_{1}+e_{k})^{T}.
$$

The Sherman-Morrison formula (cf. \cite{HenSea81}) then shows that:
$$
B_{j}^{-1}=A^{-1}+\frac{A^{-1}EA^{-1}}{1-(e_{1}+e_{k})^{T}A^{-1}(e_{1}+e_{k})}.
$$
%We can simplify this a bit. Observe first that since $C$ is circulant, the matrix $C^{-1}$ is also circulant (cf \cite[p. 73]{Circulant}). 
Now, as $A=Q(C_{k})+(s-\mu)I$, we can bring into play the observations made in Section \ref{sec:more}. We shall write $\widetilde{A}=A^{-1}$ and denote by $d$ the common value of the diagonal entries of $\widetilde{A}$.
% and also denote $\gamma=b_{1k}$.

The expression $(e_{1}+e_{k})^{T}A^{-1}(e_{1}+e_{k})$ is equal to the sum of the four corner entries of $A^{-1}$, \emph{i.e.} to $2d+2\widetilde{a}_{1k}$. Therefore:
\begin{equation}\label{eq:pc}
B_{j}^{-1}=\widetilde{A}+\frac{\widetilde{A}E\widetilde{A}}{1-2(d+\widetilde{a}_{1k})}.
\end{equation}

Taking the row sums of both sides of \eqref{eq:pc} and bearing in mind that $\widetilde{A}\mathbf{1}=\beta \mathbf{1}$, we obtain that the vector of row sums of $\widetilde{A}E\widetilde{A}$ is equal to $2\beta(\widetilde{A}e_{1}+\widetilde{A}e_{k})$. Therefore:
\begin{equation}\label{eq:ri1}
r_{i}(B^{-1}_{j})=
%\begin{cases}
\beta \left(1 + \frac{2(d+\widetilde{a}_{1k})}{1-2(d+\widetilde{a}_{1k})} \right), i \in \{1,k\}  
\end{equation}
and
\begin{equation}\label{eq:ri2}
r_{i}(B^{-1}_{j})=\beta \left(1 + \frac{2(\widetilde{a}_{i1}+\widetilde{a}_{ik})}{1-2(d+\widetilde{a}_{1k})} \right), 2 \leq i \leq k-1. 
%\end{cases}
\end{equation}

Finally, we use Lemma \ref{lem:bai} to estimate $d,\widetilde{a}_{1k},\widetilde{a}_{i1},\widetilde{a}_{ik}$: since $\lambda=s-\mu>2$, we have that $d<0.3$ and $|\widetilde{a}_{1k}|,|\widetilde{a}_{i1}|,|\widetilde{a}_{ik}| <0.1$. Furthermore, a simple determinantal calculation shows that $\widetilde{a}_{1t}<0$. This implies that $$\frac{2(d+\widetilde{a}_{1k})}{1-2(d+\widetilde{a}_{1k})}>0, \left|\frac{2(\widetilde{a}_{i1}+\widetilde{a}_{ik})}{1-2(d+\widetilde{a}_{1k})}\right| <\frac{0.4}{0.4}=1$$ and therefore $r_{i}(B_{j}^{-1})$ is always positive.
\end{proof}

\begin{rmrk}\label{rem:path}
The conclusion of Theorem \ref{thm:deg2} remains true for $s=5$ as well but the proof we gave will not go through. The difficulty is posed by the components of $G$ that are paths  and it can be handled by a different argument that uses the fact that $B_{j}$ is then tridiagonal and that therefore $B_{j}^{-1}$ is a Green's matrix (cf. \cite{Meu92}); we omit here the details, which  are somewhat tedious.

On the other hand, $S$-Rothness may fail altogether when $s=4$. Indeed, it can be easily verified that, say, $\overline{K}_{4} \vee P_{60}$ is not $S$-Roth.
\end{rmrk}

\section{Open problems and remarks}\label{sec:open}
\begin{enumerate}
\item
An attractive informal way of describing Corollary \ref{cor:st} is: when $B_{H,S}$ is the complete bipartite graph, you can create any graph $G$ you like on the vertices in $T$ and still be sure that $H$ will be $S$-Roth. 

It is then natural to ask whether this property holds when $B_{H,S}$ is not quite complete bipartite but almost so. For example, suppose that $B_{H,S}$ is $K_{s,t}$ minus one edge (say, the edge $k_{0}i_{0}$ for some $k_{0} \in S,i_{0} \in T$). We can apply Theorem \ref{thm:harmcond} to show that in this situation $H$ will be $S$-Roth irrespective of $G$ wherever $s \geq t+1$. 

To see this, observe that $$\sum_{k \in N_{ii_{0}}}{\frac{1}{d_{B}(k)}}=\frac{s-1}{t}$$ and that $$\sum_{k \in N_{ij}}{\frac{1}{d_{B}(k)}}=\frac{s-1}{t}+\frac{1}{t-1}, \quad i_{0} \notin \{i,j\}.$$
%Observe also that an application of Corollary \ref{cor:bdeg} would have only allowed us to conclude this for $s \geq t+2$.

However, empirical evidence (we tested $s=t=6$ and $s=t=8$) suggests that $H$ may be $S$-Roth even for $s=t$. Therefore, we can pose the following problem:

%, a case that, as we have just seen, is not covered by Theorem \ref{thm:harmcond}. 

\begin{prob_s}
%For a fixed dense biparite graph $B$ on $s+t$ vertices find the smallest positive $f(B)$ such that every $H$ graph with $B_{H,S}=B$ with $s \geq t+f(B)$ must be $S$-Roth, irrespective of $G_{H,T}$.
Describe the dense bipartite graphs $B$ with biparititions of size $s,t (s \geq t)$ such that every graph $H$ with $B_{H,S}=B$ is $S$-Roth, irrespective of the structure of $G_{H,T}$. Such graphs can be called \emph{Ultra-Roth}.
\end{prob_s}

Corollary \ref{cor:st} asserts that $K_{s,t}$ is Ultra-Roth wherever $s \geq t$ and we have seen before that $K_{s,t}-e$ is Ultra-Roth wherever $s \geq t+1$. 
%(and conjectured that it stays Ultra-Roth even when $s=t$.

\item
We offer a conjecture that is inspired by Theorem \ref{thm:deg2} and supported by extensive empirical evidence:

\begin{conj_s}\label{conj:tree}
Let $t>s \geq 6$ and let $H=\overline{K}_{s} \vee G$. If $G$ is a tree and if $\Delta(G) \leq s$, then $H$ is $S$-Roth.
\end{conj_s}

It might be possible to attack Conjecture \ref{conj:tree} using the methods of Section \ref{sec:q_alt}, in a way that was hinted at in Remark \ref{rem:path}. The representation of the inverse of a tridiagonal matrix as a Green's matrix has known generalizations to the case of a tree \cite{Nab01} and therefore $R_{\mu}^{-1}$ is, in principle, amenable to analysis. However, such an approach has required great delicacy even in the case of a path and the case of a general tree is likely to present greater technical difficulties in the analysis of the inverse. 

\item
An even bolder conjecture has been verified by us for values of $t$ up to $8$ (using again data from \cite{RoyleBipData}):
\begin{conj_s}\label{conj:s}
Let $t>s \geq 6$ and let $H=\overline{K}_{s} \vee G$. If $\Delta(G)<s$, then $H$ is $S$-Roth.
\end{conj_s}
At the moment we cannot suggest a possible approach to Conjecture \ref{conj:s}. 

\item
As we have seen in Section \ref{sec:case}, for all its combinatorial elegance Theorem \ref{thm:harmcond} covers only a relatively small part of the situations where $S$-Rothness arises. It might be possible to obtain another theorem of this kind by using a little-known result of Gavrilov:

\begin{thm}\cite{Gav01}
Let $M \in \mathbb{R}^{n \times n}$ be a symmetric positive definite matrix. If for some $2 \leq m <n$ all principal submatrices of order $m$ of $M$ are monotone, then $M$ is monotone.
\end{thm}

This theorem generalizes the well-known fact (cf. \cite[pp. 134-138]{Avi}) that $M$-matrices are monotone, since for a positive definite matrix the nonpositivity of offdiagonal entries is equivalent to the monotonicity of principal submatrices of order $2$.

Since $Q_{\mu}$ is positive definite, it might be possible to obtain a counterpart to our Theorem \ref{thm:harmcond} by formulating a combinatorial condition that ensures the monotonicity of the principal submatrices of order $3$ of $Q_{\mu}$.

\end{enumerate}

\section*{Acknowledgments}
We are grateful to the anonymous referee for very useful comments.

\bibliographystyle{abbrv}
\bibliography{nuim}

\end{document}